\documentclass{article}
\usepackage {CJKutf8}
\usepackage{amsmath}
\usepackage{amssymb}
\usepackage{amsthm}
\usepackage{amsrefs}
\usepackage{hyperref}

\usepackage{todonotes}

\theoremstyle{plain}%
\newtheorem{theorem}{Theorem}
\newtheorem{proposition}[theorem]{Proposition}%
\newtheorem{lemma}[theorem]{Lemma}
\newtheorem{corollary}[theorem]{Corollary}

\theoremstyle{remark}%
\newtheorem{remark}{Remark}%

\theoremstyle{definition}%
\newtheorem{definition}{Definition}%

\numberwithin{equation}{section}
\numberwithin{theorem}{section}

\newcommand{\inn}[2]{\left\langle#1,\,#2\right\rangle}

\newcommand{\grad}{\nabla}
\newcommand{\Lap}{\Delta}
\newcommand{\di}{\partial}

\DeclareMathOperator{\Tr}{Tr}

\newcommand{\br}[1]{\left\langle#1\right\rangle}
\newcommand{\si}{\sigma}
\newcommand{\eps}{\epsilon}
\newcommand{\ls}{\lesssim}
\newcommand{\g}{\gamma}
\newcommand{\al}{\alpha}
\newcommand{\Si}{\Sigma}

\newcommand{\Cb}{\mathbb{C}}
\newcommand{\Xc}{\mathcal{X}}
\newcommand{\Mc}{\mathcal{M}}
\newcommand{\Ac}{\mathcal{A}}
\newcommand{\Zb}{\mathbb{Z}}
\newcommand{\Rb}{\mathbb{R}}

\newcommand{\one}{\ensuremath{\mathbf{1}}}

\newcommand{\Sc}{\ensuremath{\mathcal{S}}}

\newcommand{\abs}[1]{\ensuremath{\left\lvert#1\right\rvert}}
\newcommand{\norm}[1]{\ensuremath{\lVert#1\rVert}}

\newcommand{\Set}[1]{\left\{#1\right\}}
\newcommand{\fullfunction}[5]{\ensuremath{
		\begin{array}{ccrcl}
			{#1}    & \colon  & {#2} & \longrightarrow & {#3} \\
			\mbox{} & \mbox{} & {#4} & \longmapsto     & {#5}
\end{array}}}

\newcommand{\od}[2]{\ensuremath{\frac{d#1}{d#2}}}
\newcommand{\md}[6]{\ensuremath{
		\ifinner
		\tfrac{\partial{^{#2}}#1}{\partial{#3^{#4}}\partial{#5^{#6}}}
		\else
		\tfrac{\partial{^{#2}}#1}{\partial{#3^{#4}}\partial{#5^{#6}}}
		\fi
}}
\newcommand{\del}[1]{\left(#1\right)}
\newcommand{\thmref}[1]{Theorem~\ref{#1}}

\newcommand{\defnref}[1]{Definition~\ref{#1}}
\newcommand{\secref}[1]{Section~\ref{#1}}
\newcommand{\lemref}[1]{Lemma~\ref{#1}}
\newcommand{\propref}[1]{Proposition~\ref{#1}}
\newcommand{\remref}[1]{Remark~\ref{#1}}

\newcommand{\corref}[1]{Corollary~\ref{#1}}

\newcommand{\wn}{\ensuremath{w^{(0)}}}
\newcommand{\sn}{\ensuremath{\si^{(0)}}}
\newcommand{\wo}{\ensuremath{w^{(1)}}}
\newcommand{\so}{\ensuremath{\si^{(1)}}}
\newcommand{\Un}{\ensuremath{U^{(0)}}}
\newcommand{\Uo}{\ensuremath{U^{(1)}}}
\begin{document}
	
		
		
		
		\title{A generic framework of adiabatic approximation for nonlinear evolutions II}
		
			\author{Jingxuan Zhang 
			(\begin{CJK*}{UTF8}{gbsn}张景宣
			\end{CJK*})$^{1,2}$
			\footnote{Email: \url{jingxuan.zhang@math.ku.dk}}
		}
		\date{%
			$^1$	University of Copenhagen\\
			Department of Mathematical Sciences\\
			Universitetsparken 5\\
			2100 Copenhagen, Denmark\\[2ex]
			$^2$	University of Toronto\\
			Department of Mathematics\\
			40 St George Street\\
			Toronto, ON, M5S 2E4, Canada\\[2ex]
			\today
		}
		
		\maketitle
		
		\begin{abstract}
In this paper, we continue the development of a generic adiabatic scheme for nonlinear evolutions. We consider an abstract gradient flow of some energy functional, together with a given manifold of static solutions arising from broken symmetries. First, we list a number of explicit and generic conditions on the energy functional that ensures the validity of our adiabatic scheme. Then, we construct some explicit low-energy but no necessarily static configurations, which form a stable manifold with finite codimensions for the given gradient flow. Thirdly, we show that the gradient flow is globally well-posed with initial configuration from the stable manifold. Finally, we show that any solution to the full gradient flow starting from the stable manifold is essentially governed by an effective equation on the manifold of static solutions, up to a uniformly small and dissipating error term.
		\end{abstract}

		\section{Introduction}
		Consider the (negative) gradient flow
		\begin{equation}
			\label{1}
			\di_tu=- E'(u).
		\end{equation}
		Here $u=u(t)\in U,\,t\ge 0$ is a $C^1$ path of vectors in some open set $U\subset X'$ in a real Hilbert
		space $X'$, sitting in a larger Hilbert space $X$. 
		The map $E:U\subset X'\to \Rb$ is some energy functional which is $C^2$ on $U$. The vector $E'(u)\in X$ is the gradient of $E$ at $u$ w.r.t. the inner product on $X$. See Appendix for relevant definitions. 
		
		Throughout this paper, we assume the energy functional $E$ has a non-trivial symmetry Lie group, and there exists a static solution $u\equiv u_0\in X$ to \eqref{1} (i.e. an equilibrium state), which simultaneously breaks a symmetry Lie subgroup $\Si$ with $1\le \dim\Si<\infty$. We call this $\Si$ the \textit{reduced symmetry (Lie) group} for \eqref{1}.
		
		If we shrink $U$ to be a small ball around $u_0$, then the standing assumptions above give rise to a finite-dimensional manifold of static solutions (i.e. ensemble of equilibrium states)
		$$\Mc_0:=f(\Si)\cap U,$$ 
		where $f:\Si\to X'$ is an immersion of the reduced symmetry Lie group, given in \eqref{2.6} below. See \secref{sec:2.1} for details on the set up for \eqref{1}. 
		
		For the main result of this paper, \thmref{thm1}, we prove that there exists a finite-codimensional manifold $\Mc$ with
		$$\Mc_0\subset \Mc\subset U,$$
		s.th. the full dynamics \eqref{1} with initial configurations from $\Mc$ essentially reduces to an explicit effective (i.e.  adiabatic) dynamics, namely \eqref{2.23}, posed on the finite-dimensional manifold of static solutions  $\Mc_0$.
		
		In a previous paper \cite{EffDyn}, under a number of generic sufficient conditions, we have developed an adiabatic approximation scheme for \eqref{1}. There, we have assumed global well-posedness for \eqref{1} with initial configurations in $U$, as well as the existence of what we call \textit{approximate solitons}. In this paper, we present a new set of sufficient conditions on the linearized operator of the map $u\mapsto E'(u)$ at the manifold $\Mc_0$. These conditions are similar but less restrictive than the ones from \cite{EffDyn}. The larger manifold $\Mc$ plays a similar role as the manifold of approximate solitons in \cite{EffDyn}, except for the key difference that the effective equation is not posed on $\Mc$, but rather on the smaller manifold $\Mc_0$. See \secref{sec:2.3} for a discussion. 
		
		For motivations and backgrounds, we refer the readers to \cites{EffDyn,AdiabMCF}. We single out a paper by W. Schlag \cite{MR2480603} from which we have drawn much inspiration. An application of the abstract theory developed in this paper is given in \cite{AdiabMCF}, where we study the dynamical stability of the singularities of the mean curvature flow.
		
		\subsection{Arrangement}
		This paper is organized as follows. In \secref{sec:2.1}, we discuss in details the set up for the evolution \eqref{1}. In \secref{sec:2.2}, we present the main assumptions of this paper, \eqref{C1}-\eqref{C4}. Then we state the main result in \secref{sec:2.3}.
		
	The main result, \thmref{thm1}, is proved in two steps in Sections \ref{sec:3}-\ref{sec:6}. The upshot is to eliminate the destabilizing effect of the zero-unstable modes of the linearized operator at the manifold of trivial static solution, see \eqref{L}. 
	
	At the first order, in \secref{sec:3},we eliminate the zero modes using modulation method. Yet this is not enough due to the presence of unstable modes, see \eqref{C2}. Such modulation method is customary in the analysis of solitary wave dynamics in hyperbolic systems, see \cites{MR820338, MR1071238,MR1170476,MR2094474}.
	
	At the next order, we establish a fixed point scheme to eliminate the unstable modes that have persisted the modulation. In \secref{sec:4}, we show that \thmref{thm1} is equivalent to a certain fixed point problem (see \propref{prop4.1}). In \secref{sec:5}, we solve the fixed point problem posed in the preceding section. This gives rise to the set $\Mc$ mentioned above. In \secref{sec:6}, we prove certain smallness estimates on $\Mc$. These estimate then ensure that $\Mc$ form a non-degenerate stable manifold with finite codimensions.  
	
	The second part of the paper can be viewed as constructing a stable manifold to \eqref{1}, and compared to the established invariant manifold theory for dissipative systems, see e.g. \cites{MR1000974,MR1160925,MR1445489,MR1675237,MR2439610}. Here we emphasize that our result gives global-wellposedness as a by product (in contrast to the cited works above), and produce an explicit and tractable effective dynamics for \eqref{1}, namely \eqref{2.23}.

		\subsection*{Notations} Throughout this paper, when no confusion arises, we shall drop the time dependence $t$ in subscripts. 
		An estimate $A\ls B$ means there is some $C>0$ independent 
		of time and all the parameters in question, s.th. $A\le CB$. The expression $A\sim B$ means that
		$A\ls B$ and $B\ls A$ hold simultaneously. The summation convention for repeated indices is always understood.
		
		\section{Set up and the main results}
		In this section, we explain the set up for a generic adiabatic scheme for \eqref{1}. Then we present a list of sufficient conditions for the validity of this scheme, and justify their generality. These assumptions consist of some secondary ones, (f1)-(f2), (h1)-(h2) from \secref{sec:2.1}, which concern with generic properties for \eqref{1} in the presence of symmetry. 	In \secref{sec:2.2}, we present a group of non-trivial assumptions, \eqref{C1}-\eqref{C4}, in terms of certain linearized operators. Lastly, in \secref{sec:2.3}, we state the main result, \thmref{thm1}, under these assumptions, and discuss its implications. 	This theorem is proved in \secref{sec:3}-\secref{sec:6}.
		
		\subsection{Set up}\label{sec:2.1}
			Let $X$ be a Hilbert space with norm $\norm{\cdot}_X$ induced by the inner product $\inn{\cdot}{\cdot}_X$. Let $X'$ be a subspace of $X$, together with a stronger norm $\norm{\cdot}_{X'}$. We understood that $X'$ inherits the inner product $\inn{\cdot}{\cdot}_X$. 
			
			A typical example is $X=L^2(M,d\mu)$, the Lebesgue space of order $2$ over a measure space $(M,d\mu)$, and $X':=H^s$, the Sobolev space of order $s\ge0$. The inner product on both spaces are the $L^2$-one, but $X'$ is equipped with the stronger $H^s$-norm. 

			Let $E:X\to\Rb$ be a $C^2$ energy functional in the the sense that there exists an open set $U\subset X'$ (w.r.t. the topology induced by $\norm{\cdot}_{X'}$) s.th. the map $E':U\subset X'\to X$, sending a vector $u$ to the $X$-gradient $E'(u)$, is well-defined and continuous.

			Denote  by $E''(u):X'\to X$ the Fr\'echet derivative of $E'$ at $u\in U$. Then for every $u$, $E''(u)$ is a bounded self-adjoint operator defined on all of $X'$.   
			See Appendix for relevant definitions.
			
			A typical example arising from statistical field theory  is as follows. Let $(M,g)$ be a sufficiently regular Riemannian manifold, corresponding to a physical domain of interest. Let  $X:=L^2(M,d\mu;\Rb)$, where $d\mu$ is the canonical measure induced by $g$. This $X$ describes a space of order parameter. We take $X'=H^1(M,d\mu;\Rb)$, and define 
$$E(u)=\int \abs{\grad u}^2+G(u),$$
where $G:\Rb \to \Rb$ is an explicit nonlinear function. This $G$ incorporates certain nonlinear phenomen, e.g.  due to mean field approximation.
Suppose there is a critical point $u_0\in X'$ to $E$ with $\abs{E(u_0)}<\infty$, say, describing a pure state, which amounts to a trivial static solution to \eqref{1}. Suppose $G$ satisfies certain regularity assumptions depending on this $u_0$. Then one can show using Sobolev inequalities that $E$ is $C^2$ in a neighbourhood $U\subset X'$ around $u_0$.
			
We assume evolution \eqref{1}
has a finite-dimensional \textit{full symmetry Lie group} $G$ with $\dim G\ge1$.
Precisely, we assume each element $g\in G$ acts on the configuration space $X$ as a linear transformation
\begin{equation}\label{2.0}
	T(g)\in Aut(X),
\end{equation}
satisfying
\begin{equation}
	\label{2.1}
	E(T(g)u)=T(g)E(u)\quad (u\in X, g\in G).
\end{equation}
In applications, such $T(g)$ can be translation, rotation, gauge transform, etc..
Denote by  
\begin{equation}\label{T}
	T:g\mapsto T(g)
\end{equation} the representation of $G$ via automorphism of $X$.

	We assume there exists some $u_0\in X'$ s.th. $E'(u_0)=0$. In applications, this $u_0$ is usually a well-understood ground state. Denote by  $H\subset G$ the maximal closed normal subgroup of $G$ s.th. $T(g) u_0= u_0$ 
for every $g\in H$,
and define the 	\textit{reduced symmetry Lie group} by
\begin{equation}\label{2.4}
	\Si:=G/H.
\end{equation}
We assume $$\dim \Si\ge1.$$ 
Physically,  this means that $u_0$ break parts of the full symmetry group of the energy functional $E$ (i.e. simultaneous symmetry breaking).

In the remaining of this paper, we will only be concerned with the reduced symmetry group $\Si$, instead of the full symmetry group $G$. 

We assume there exists a Riemannian metric  $g$ on $\Si$, which turns the latter into a Riemannian manifold. For example, $\Si$ is a matrix Lie subgroup, then one can take $g(A,B)=\Tr A^*B$. Denote by $$d(\si,\si'),\quad\si,\si'\in\Si$$ the distance induced by this metric $g$ on $\Si$, and $$\norm{\cdot}_Y$$ the norm induced by $g$ on the tangent space $T_\si\Si$. For simplicity, we assume $\norm{\cdot}_Y$ does not change as the base point $\si$ varies in $\Si$.

Define a map
\begin{align}\label{2.6}
\fullfunction{f}{\Si}{X}{\si}{T(\si)u_0}
\end{align}
We assume the map $f$ has the following properties:
\begin{enumerate}
	\item[(f1)] $f(\one_\Si)=u_0$,  $f(\Si)\subset X'$, and $f$ is a $C^2$ immersion as a map from $\Si\to X'$.
	\item[(f2)] There exists $c_f>0$ s.th.
	\begin{equation}
		\label{f2}
		\norm{f}_{C^2}:=\sup_{\si\in\Si_0}\del{\norm{f(\si)}_{X'}+\norm{df(\si)}_{Y\to X'}+\norm{d^2f(\si)}_{Y\to L(Y,X')}}\le c_f\tag{f2}.
	\end{equation} 
\end{enumerate}
Condition (f1) allows one to view $f(\Si)$ as a manifold in $X'$.
Condition (f2) is a technical assumption used to derive various uniform estimates (see the remark below). These conditions are analogous to the one in \cite{EffDyn}*{Sec 2.2}. In applications, (f1)-(f2) mostly follow easily from the explicit form of symmetry $T(\si)$ in \eqref{2.0}.

\begin{remark}
	If $f$ satisfies condition (f1), then there exists an open connected submanifold \begin{equation}
	\label{2.7}\Si_0\subset \Si
\end{equation} around the identity element $\one_\Si$, s.th. \begin{equation}\label{M0}
\Mc_0:=f(\Si_0)\subset U
\end{equation} forms a finite-dimensional immersed manifold, and \eqref{f2} holds. This gives rise to a manifold of static solution to \eqref{1}. In the setting of \cite{EffDyn}, this $f(\Si_0)$ amounts to a manifold of \textit{exact} solitons.
\end{remark}

Shrinking $\Si_0$ if necessary, there exists a local trivialization
\begin{equation}
	\label{h}
	h:\Rb^{\dim \Si}\times \Si_0 \to \bigcup_{\si\in\Si_0}T_\si\Si.
\end{equation}
This allows us to write a typical element in $T_\si \Si\cong Lie(\Si)$ as  $$a^i\xi_i(\si),$$ where 
\begin{equation}\label{xi}
	\xi_i(\si):=h(E_i,\si)\quad (1\le i \dim \Si),
\end{equation} and $\Set{E_i}$ is a basis of $\Rb^{\dim \Si}.$

We assume the trivilization $h$ satisfies 
\begin{enumerate}
	\item[(h1)] For every $\si\in\Si_0$, there holds the orthogonality condition  
	\begin{equation}\label{h1}
		\inn{df(\si)\xi_i(\si)}{df(\si)\xi_j(\si)}_X=\delta_{ij}\tag{h1}.
	\end{equation}
	\item[(h2)] There exists $c_h>0$ independent of $i$ s.th.
	\begin{equation}
		\label{h2}
		\norm{\xi_i(\cdot)}_{C^1}:=\sup_{\si\in\Si_0}\norm{\xi_i(\si)}_Y+\norm{d\xi_i(\si)}_{Y\to Y }\le c_h\tag{h2}
	\end{equation}
Recall that $\norm{\cdot}_Y$ denotes the norm on $T_g\Si$ induced by the Riemannian metric $g$ on $\Si$. 
\end{enumerate}
The role of condition \eqref{h1} will be clear after \defnref{defn1} below. Condition \eqref{h2} plays a similar role as \eqref{f2}. For bounded $\Si_0\subset \Si$, this  follows from the regularity assumption (h1).

Now we can define a class of key objects for the subsequent analysis.
\begin{definition}[symmetry zero modes]
	\label{defn1}
A vector $\phi\in X'$ is called a symmetry zero mode of the linearized operator $L(\si)$	if and only if 
	\begin{equation}
		\label{szm}
		\phi=a^idf(\si)\xi_i(\si)
	\end{equation}
	for some $a^i\in\Rb,\,1\le i\le \dim \Si$. 
\end{definition}
\begin{remark}
	Condition \eqref{h1} ensures that these modes are orthonormal in $X$ as $\si$ varies in $\Si_0$.
\end{remark} 

Denote by 
\begin{equation}
	\label{L}
	L(\si):=E''(f(\si))
\end{equation}
the linearized operator of $E'$ at $f(\si)$. This is well-defined and continuous in $\Si_0$, by the assumption (f1) on $f(\si)$ and the $C^2$ regularity assumption of $E$ on $U$.

\begin{lemma}
	\label{lem2.1}
	If $\phi$ is a symmetry zero mode of $L(\si)$, then $L(\si)\phi=0$. 
\end{lemma}
\begin{proof}
	Write $\phi=a^idf(\si)\xi_i(\si)$ as in \eqref{szm}.
	Take a path $\si(s)\in\Si,\,\abs{s}\ll1$ with $\si(0)=\si,\,\di_s \si(0)= 
	a^i\xi_i(\si)$ in terms of the basis \eqref{xi}.
	By \eqref{2.1}, we have 
	\begin{equation}\label{2.9}
		E'(f(\si(s)))\equiv 0.
	\end{equation}
	Differentiating this w.r.t. $s$ and using the chain rule, we find 
	$$0= \di_s E'(f(\si(s)))= E''(f(\si(s)))df(\si(s))\di_s\si(s).$$
	By the initial condition for $\si(s)$, evaluating this at $s=0$ yields $L(\si)\phi=0$. 
\end{proof}

\begin{remark}
	We note a small caveat here regarding the \defnref{defn1}. It is not necessary for the symmetry modes to be in $X'$, which is usually taken to be the domain of $L(\si)$. For example, the Ginzburg-Landau vortices $$u^{(n)}:\Rb^2\to \Cb,n\in\Zb $$
	are solution to the semilinear elliptic equation
	$$-\Lap u +(\abs{u}^2-1)u=0,$$
	with $\deg u^{(n)}=n$. This equation is naturally posed on $L^2=L^2(\Rb^2,\Cb)$, and has translation symmetry, which generates the symmetry zero modes
	$$\di_{x^j}u^{(n)},\quad j=1,2$$
	as in \defnref{defn1}. Yet, for every $n\ne0$, these zero modes are not in the natural configuration space $L^2$ (though they are in $L^p$ for every $p>2$).
	
\end{remark}


As we show in \secref{sec:3}, symmetry zero modes can be easily eliminated by a first order correction, known as the \textit{modulation equation}. This method was first introduced in \cites{MR1071238,MR1170476} to study a class of nonlinear Sch\"odinger equations. It effectively eliminates the influence of $\xi_i$ from the dynamics.

\subsection{Main assumptions}\label{sec:2.2}

The main assumptions in this paper are concerned with the linearized operator at 
\begin{equation}
	\label{2.11}
	L(\si)=E''(T(\si)u_0):X'\subset X\to X,
\end{equation}
defined in \eqref{2.7} above. 

\textit{Main assumptions}. Fix $\si\in\Si_0$ with $\Si_0\subset \Si$ given by a neighbourhood around $1_\Si$, see \eqref{M0}.
\begin{itemize}
	\item[(C1)] $L(\si)$ is self-adjoint, and there exists $c_1>0$ independent of $\si$ s.th. \eqref{2.11} is uniformly bounded as \begin{equation}\label{C1}
		\norm{L(\si)}_{X'\to X}\le c_1.\tag{C1}
	\end{equation}

	\item[(C2)]  $0$ is an embedded eigenvalue of $L(\si)$, and the associated eigenvectors consist solely of the symmetry zero modes \eqref{szm}. The discrete spectrum $\si_{\rm d}(L(\si))$ consists of isolated eigenvalues 
	\begin{equation}
		\label{C2}
		E_1(\si)\le E_2(\si)\le \ldots \le E_N(\si)\le E<0.\tag{C2}
	\end{equation}
	for some $N\ge0$ and $E<0$ independent of $\si$. The multiplicity of $E_p(\si),\,0\le p\le N$ is finite and independent of $\si$.

	We denote by $m_p$ the multiplicity of $E_p(\si)$, so that  $1\le m_p<\infty$. 
	The eigenvectors associated to $E_j(\si)$ are denoted by \begin{equation}\label{usm}
		\phi_{j_p}(\si),\quad 1\le j_p\le m_p.
	\end{equation}

	\item[(C2)] There holds the propagator estimate
\begin{equation}
	\label{C3}
	\norm{e^{-tL(\si)}P_S(\si)}_{X\to X'} \le c_2\br{t}^{-\al} \quad (t\ge0 )\tag{C3},
\end{equation}
 for some $\al\ge 1,\, c_2>0$ independent of $\si,\,w,\,t$. Here $P_S(\si)$ denotes the orthogonal projection onto the stable range
\begin{equation}
	\label{2.14}
	\Sc(\si):=\Set{w\in X: \inn{w}{\phi_{j_p}(\si)}_X=0\text{ for every }1\le j_p\le m_p,\; 0\le p \le N}
\end{equation}
	\item[(C4)] 
	There exists $0< \delta_0\ll1$ s.th. the for every $\si,\si'\in\Si$ with $d(\si,\si')\le\delta_0$, there holds
	\begin{equation}
		\label{C4}
		\norm{L(\si)-L(\si')}_{X'\to X}\le\delta_0d(\si,\si').\tag{C4}
	\end{equation}
\end{itemize}

A number of remarks are in order to justify these conditions.
\begin{remark}
	Condition (C1) is customary for equation arising from physics. For example, in applications to classical field theory, $L(\si)$ is usually a Schr\"odinger operator of the form $L(\si)=-\Lap+V(\si),$
	where $V(\si):X\to X$ is a multiplication operator, arising from the linearization process. 
\end{remark}
\begin{remark}
	Condition \eqref{C2} is customary for spectral analysis of Schr\"odinger operators, e.g. \cites{MR1639709,MR1639713}. By \cite{MR2006008}*{thm(D)}, if $E$ is $C^3$ on $U$, then \eqref{C2} holds for all $\si\in\Si_0$. 
\end{remark}
\begin{remark}
	Condition \eqref{C3} compensates the presence of unstable modes of $L(\si)$. For fixed $\si$ and all large $t$, \eqref{C3} holds, for example, if $0\in\si_{\rm ess}$ is not a resonance, in the sense that $L(\si)P_S(\si)\ge c>0$. For a discussion of the latter spectral gap condition, see \cite{EffDyn}*{Sect. 2.2}.
\end{remark}
\begin{remark}
	Condition (C4) can be stated independent of the linearized operator $L(\si)$. Indeed, the point here is to require Fr\'echet derivatives of $E$ to vanish at a critical point $u_*$ at least up  to \textit{the second order}. This holds, for example, if $E$ is analytic at the critical point $u_*$.  
\end{remark}
\begin{remark}\label{rem4}
	Condition \eqref{C4} should be understood as an assumption on the nonlinearity. Indeed, for $\si\in\Si,\,w\in X'$ define the nonlinearity 
	\begin{equation}
		\label{2.15}
		N(\si,w):=E'(f(\si)+w)-L(a)w.
	\end{equation}
If $E$ is $C^2$ on $U$, then $N(\si,w)$ is continuous on $\Si_0$ times a small ball in $X'$, and satisfies $\norm{N(\si,w)}_X=o(\norm{w}_{X'})$. In principle, one would like to upgrade this remainder estimate to a quadratic one of the form $\norm{N(\si,w)}_X\ls\norm{w}_{X'}^2$, where the implicit constant is independent of $\si$. However, in application, this is sometimes difficult when one uses non-standard norms with strong decay properties. For example, in \cite{MR3374960}, the nonlinearity fails to be $O(\norm{w}_{X'}^{1+\eps})$ for any $\eps>0$ due to the use of Gaussian weighted Sobolev norm. See Sect. 4 of that paper. 
\end{remark}
\begin{remark}
	In the sequel, we will mostly use the following consequence of \eqref{C4}:   There exists $0<\delta=\delta(c_f,\delta_0)\ll1$ (see \eqref{C4} and \eqref{f2}) s.th. for every $\si,\,\si'\in\Si_0$ and $w,w'\in X'$ with $\norm{w}_{X'}+\norm{w}_{X'}\le\delta$, 
there holds
	\begin{equation}
		\label{2.17}
		\norm{N(\si,w)-N(\si',w')}_X\le \delta(d(\si,\si')+\norm{w-w'}_{X'}).
	\end{equation}
\end{remark}

\subsection{Main result}\label{sec:2.3}

In the remaining of this paper, we fix some $\si_0\in\Si_0$. In application, this choice amounts to a stationary symmetry frame for \eqref{1}. For $r>0$ we write $$B_r(\eta):=\Set{\eta'\in X':\norm{\eta-\eta'}_{X'}<r},\quad B_r\equiv B_r(0).$$ 

The main result of this paper is the following statement about the key map $\Phi$ defined in \eqref{Phi}.
\begin{theorem}
	\label{thm1}
	Suppose (f1)-(f2), (h1)-(h2), and (C1)-(C4) hold. Then there exist $0<\delta\ll1$, $c_0>0$ independent of $\delta$, together with a map 
	\begin{equation}
		\label{2.18}
		\Phi:\Sc(\si_0)\cap B_{c_0\delta}\to X'\quad \text{defined in \eqref{Phi} below, }
	\end{equation}
	s.th. for every $\eta,\,\eta'$ in the domain of $\Phi$,
	\begin{align}
	\label{2.19}
	\norm{\Phi(\eta)}_{X'}=&o(\norm{\eta}_{X'})\quad(\delta\to0),\\
	\label{2.20}
	\norm{\Phi(\eta)-\Phi(\eta')}_{X'}=&O(\delta\norm{\eta-\eta'}_{X'})\quad (\delta\to0),
\end{align}
and the following holds:
\begin{enumerate}
	\item (Global existence) For every $\eta\in\Sc(\si_0)$ with $\norm{\eta_0}_{X'}<c_0\delta$, there exists a unique global solution $$u\in C^1 (\Rb_{\ge0},\Si\times X)\cap Lip (\Rb_{\ge0},\Si\times X')$$
	to the gradient flow \eqref{1}, with initial configuration
\begin{equation}\label{ic}
		u(0)=f(\si_0)+\eta+\Phi(\eta).
\end{equation}
	\item (Effective dynamics) For all $t\ge0$, this solution
	$u$ is of the form
	\begin{equation}
		\label{2.21}
		u=f(\si)+w,
	\end{equation}
	where $f$ is defined in \eqref{2.6}, the path $w=w(t),\,t\ge0$ satisfies
	\begin{equation}
		\label{2.22}
		\norm{w(t)}_{X'}\le \delta \br{t}^{-\al},
	\end{equation}
and $\si(t)\in\Si,\,t\ge0$ evolves according to the effective dynamics
\begin{equation}
	\label{2.23}
	\dot\si=h(a(\si,w),\si),\quad \si(0)=\si_0.
\end{equation}
Here, $h$ is the trivilization given in \eqref{h}, the vector $a=a(\si,w)\in\Rb^{\dim\Si}$ is the unique solution to the equation $G(a,\si,w)=0$, with
$G$ defined in \eqref{3.12}. Moreover, for all $t\ge0$, there holds
\begin{equation}
	\label{2.24}
	\norm{\dot\si(t)}_Y=O(\delta\br{t})\quad (\delta\to0).
\end{equation}
\end{enumerate}
\end{theorem}

\begin{definition}[stable manifold]\label{sm}
	With notations as in \thmref{thm1} above, the set
	\begin{equation}
		\label{Mc}
		\Mc=\Mc(\si_0,\delta):=\Set{f(\si_0)+\eta+\Phi(\eta):\eta\in\Sc(\si_0),\,\norm{\eta}_X'<c_0\delta}
	\end{equation}
	is called the stable manifold for \eqref{1},
\end{definition}
This set $\Mc$ forms a manifold. Indeed, by Rademacher's theorem \cite{MR1373829}, Lipschitz estimate \eqref{2.20} ensures that $\Phi$ is differentiable a.e. with respect to a suitable Radon measure on $X'$, and the Fr\'echet derivative $d\Phi(\eta):X'\to X'$ satisfies $\norm{d\Phi(\eta)}_{X'\to X'}\ls\delta$ for fixed small $\delta$. This, together with definition \eqref{Mc} above, implies that $\Mc$ is an a.e. immersed manifold in $X'$.

\begin{remark}\label{rem:gen}
	By \eqref{2.19}, the set $\Mc$ is a small graph over the  finite-codimensional affine space $f(\si_0)+\Sc(\si_0)$. By the preceding remark, this graph can be viewed as a finite-codimensional manifold inside the small ball $B_\delta(f(\si_0))\subset U$. This justifies the generality of \thmref{thm1}, as Items 1-2 above only have to do with initial configuration from $\Mc$.
\end{remark}

We understand \thmref{thm1} in two ways. First, in physical terms, this can be viewed as an adiabatic theorem for \eqref{1} as follows. By \eqref{2.19} and the assumption $0<\delta\ll1$ (i.e. the adiabatic condition), $\Mc$ consists of low energy solution in the configuration space for \eqref{1} (recall that $E'(f(\si_0))=0$ by \eqref{2.1}). Then  assertions \eqref{2.21}-\eqref{2.23} show that the full dynamics \eqref{1} for $u$, which a priori is an infinite-dimensional dynamical system on $X$, reduces to an effective dynamics for $\si$, which is a system of nonlinear ODEs posed on $\Rb^{\dim \Si}$. 

By construction, the manifold $\Mc_0\equiv f(\Si_0)\subset U$ of static solution is an ensemble of equilibrium states (see \eqref{M0}), and therefore \eqref{2.23} is the adiabatic approximation for \eqref{1}. The validity of this adiabatic dynamics is guaranteed, since \eqref{2.22} ensures that a generic initial perturbation (see \remref{rem:gen}) dissipates to $0$ as $t\to\infty$. 

Secondly, in the theory of invariant manifold, the set $\Mc$ can be viewed as a stable manifold for \eqref{1}. Indeed, the dissipative estimate \eqref{2.22} implies that the trivial solutions $u_0$ (or equivalently, $f(\si_0)$ for any fixed $\si_0\in\Si_0$) is asymptotically stable with initial perturbations of the form \eqref{ic}.

The remaining sections are devoted to the proof of \thmref{thm1}.

\section{First order correction}\label{sec:3}

We would like to reduce the full dynamics \eqref{1} to an  effective one on the manifold of static solution $f(\Si_0$). The main obstacle is the presence of the zero modes \eqref{szm}, and the unstable modes \eqref{usm}. In this section, we use the method of modulation equation to eliminate the symmetry zero modes. The unstable modes are handled in Sections \ref{sec:4}-\ref{sec:5}. At least heuristically, once we remove all the zero-unstable modes,  the desired dissipative estimate \eqref{2.22} would follow from   the propagator estimate \eqref{C2}.

The main result of this section is the following:
\begin{theorem}
	\label{thm3.1}
	Let $u(t)\in U,\,t\ge0$ be a solution to \eqref{1}.
	Suppose $u=u(t)$ is of the form
	\begin{equation}
		\label{3.1}
		u= f(\si)+w,
	\end{equation}
	where $\si=\si(t)$ is a path in $\Si_0$, the vector $f(\si)\in U$  (see \eqref{2.6}-\eqref{2.7}), and  $w:=u-f(\si)$  satisfies $\norm{w(t)}_{X'}\ll1$. 
	
	Suppose
	$w(0)$ is orthogonal to all symmetry zero modes of $L(\si(0))$, i.e.
	\begin{equation}
		\label{3.2}
		\inn{w(0)}{df(\si(0))\xi_i}=0\quad (1\le i\le\dim\Si).
	\end{equation}

	Then $w(t),t>0$ satisfies 
	\begin{equation}
		\label{3.3}
		\inn{w(t)}{df(\si(t))\xi_i(\si(t))}=0\quad (1\le i\le\dim\Si)
	\end{equation}
	if and only if $$\dot \si(t)=h(a(t),\si(t))=a^j(t)\xi_j(\si(t)),$$
	were $h$ is the trivialization given in\eqref{h}, and each $a^i=a^i(t)$ solves the algebraic equation
	\begin{align}
		\label{3.4}
		  a^i=&\rho_{ij}(\si, w)a^j+\tau_i(\si,w)\\
		  \label{3.4'}
		  \rho_{ij}(\si,w):=&\inn{w}{(d^2f(\si)\xi_j(\si))\xi_i(\si)}+\inn{w}{(d\xi_i(\si))\xi_j(\si)}\\
		  \label{3.4''}
		   \tau_i(\si,w):=&- \inn{N(\si,w)}{df(\si)\xi_i(\si)}.
	\end{align}

\end{theorem} 
\begin{proof}
	\eqref{3.3} holds if and only if it holds at $t=0$, i.e. \eqref{3.2} is satisfied, and 
	\begin{equation}
		\label{3.5}
		\inn{\dot w}{df(\si)\xi_i(\si)}=-\inn{w}{\di_t (df(\si)\xi_i(\si))}.
	\end{equation}
	We now calculate both sides of \eqref{3.5} explicitly, and show that it is equivalent to \eqref{3.5}.
	
	Differentiating \eqref{3.1}, we find 
	\begin{equation}
		\label{3.6}
		\dot u=df(\si)\dot \si + \dot w=a^idf(\si)\xi_j(\si)+\dot w,
	\end{equation}
	where we write the velocity $\dot\si=h(a,\si)=a^j\xi_j(\si)$ in terms of the basis \eqref{xi}. Plugging this back to the l.h.s. of \eqref{3.5}, we find
	\begin{equation}
		\label{3.7}
			\inn{\dot w}{df(\si)\xi_i(\si)}=-\inn{E'(f(\si)+w)}{df(\si)\xi_i(\si)}-a^j\inn{df(\si)\xi_j(\si)}{df(\si)\xi_i(\si)}
	\end{equation}
	Since $f(\si)\in U$ and $\norm{w}_{X'}\ll1$, the following expansion is valid: $$E'(f(\si)+w)=E'(f(\si))+L(\si)w+N(\si,w).$$
	The first term in the r.h.s. vanishes by \eqref{2.1} and the definition of $f$, see \eqref{2.9}. By \lemref{lem2.1} and the self-adjointness of $L(\si)$ from condition (C1), the inner product $$\inn{L(\si)w}{df(\si)\xi_i}=\inn{w}{L(\si)df(\si)\xi_i}=0.$$
	Hence, the first term in the r.h.s. of \eqref{3.7} simplifies to $-\inn{N(\si,w)}{df(\si)\xi_i(\si)}$. This, together with \eqref{h1}, implies that 
	\begin{equation}
		\label{3.7'}
		\inn{\dot w}{df(\si)\xi_i(\si)}=-a_i-\inn{N(\si,w)}{df(\si)\xi_i(\si)}.
	\end{equation}

	For the r.h.s. of \eqref{3.5}, we calculate 
\begin{equation}\label{3.7''}
		\begin{aligned}
		&\di_t(df(\si)\xi(\si))\\
		=&\di_t(df(\si))\xi_i(\si)+df(\si)\di_t\xi_i(\si)\\
	=&a^j((d^2f(\si)\xi_j)\xi_i+d\xi_i(\si)\xi_j)\quad (\si=a^j\xi_j).
	\end{aligned}
\end{equation}
	Here $d^2f(\si):T_\si\Si\to L(T_\si\Si, X')$ and $d\xi_i(\si):T_\si \Si, X'$ are respectively the Fr\'echet derivatives of $df$ and $\xi_i=h(\cdot, E_i)$ evaluated at $\si$, see \eqref{2.6} and \eqref{xi}. Plugging \eqref{3.7'}-\eqref{3.7''} back to \eqref{3.5} gives \eqref{3.4}.

%
%
\end{proof}

\begin{corollary}
	\label{cor3.2}
	Suppose $u=f(\si)+w$ is a solution to \eqref{1} as in \thmref{thm3.1}, and satisfies \eqref{3.2}-\eqref{3.3}. Then $\dot\si=h(a,\si)$ with $\abs{a}=o (\norm{w}_{X'})$.
\end{corollary}
\begin{proof}
	By \thmref{thm3.1}, the path $\dot \si$ satisfies 
\begin{equation}\label{3.8'}
		\dot\si=h(\si, a), \quad G(a,\si,w)=0,
\end{equation}
	where
		\begin{equation}
		\label{3.12}
		\fullfunction{G}{\Rb^{\dim \Si}\times \Si\times X'}{\Rb^{\dim \Si}}{(a,\si,w)}{a^i-\rho_{ij}(\si, w)a^j-\tau_i(\si,w)},
	\end{equation}
c.f. \eqref{3.4}-\eqref{3.4''}. We use \eqref{3.8'}-\eqref{3.12} to establish the estimate for $a$. 

	Fix $1\le i,j \le \dim \Si$. We first bound $\rho_{ij}$ from \eqref{3.4'}. 
	 The first term in $\rho_{ij}$ can be bounded by
\begin{equation}\label{3.8}
		\abs{\inn{w}{(d^2f(\si)\xi_j(\si))\xi_i(\si)}}\le  c_fc_h\dim \Si\norm{\xi}_Y \norm{w}_X,
\end{equation}
	and the second term by 
\begin{equation}\label{3.9}
		\abs{\inn{w}{(d\xi_i(\si))\xi_j(\si)}}\le c_h \dim\Si\norm{w}_X.
\end{equation}
	These follow from Cauchy-Schwartz and the assumptions \eqref{f2}, \eqref{h2}. Combining these we find $\abs{\rho_{ij}(\si,w)}\ls \norm{w}_X$. 
	
	Similarly, we find that $\tau_i$ is of sub-leading order, as
\begin{equation}\label{3.10}
		\abs{\inn{N(\si,w)}{df(\si)\xi_i(\si)}}\le c_f c_h \norm{N(\si,w)}_X=o(\norm{w}_{X'}).
\end{equation}
	The last bound follow from the regularity of $E$ on $U$. 
	
	Plugging \eqref{3.8}-\eqref{3.10} into \eqref{3.8'}, and using the definition of $G$ in \eqref{3.12}, we find that
\begin{equation}\label{3.11}
		\abs{a^i}\le c \norm{w}_X\abs{a}+o(\norm{w}_{X'})
\end{equation}
	for some absolute constant $c>0$. If $\norm{w}_{X'}\le (2c\dim \Si)^{-1}$, then \eqref{3.10} implies $\tfrac{1}{2}\abs{a}\le o(\norm{w}_{X'})$. 
	
\end{proof}

\begin{proposition}
	\label{prop3.3}
	For every $(\si,w)\in\Si_0\times X'$ with $\norm{w}_{X'}\ll1$, there exists a unique solution $a=a(\si,w)\in\Rb^{\dim\Si}$ to the equation $G(a,\si,w)=0$, where $G$ is as in \eqref{3.12}

	Moreover, there exist constants $c,\,c'>0$ independent of $(\si,w),(\si',w')$ s.th.
	\begin{align}
		\label{a1}
		\abs{a(\si,w)}\le &c\norm{w}_{X'},\\
		\label{a2}
		\abs{a(\si,w)-a(\si',w')}\le& c' (d(\si,\si')+\norm{w-w'}_{X'}).
	\end{align}
\end{proposition}
\begin{proof}
	By formula \eqref{3.4'}-\eqref{3.4''}, we find that equation \eqref{3.12} has trivial solution $(a,\si,w)=(0,0,0)$.
	As a part of the standing assumptions, $f$ is $C^2$ and  $\xi_i$ is $C^1$ on $\Si_0$, and $N(\si,w)$ is continuous from $X'\to X$ for $\abs{w}_{X'}\ll1$. Hence the map
	$G(a,\si,w)$ is $C^1$ in $a$ for fixed $(\si,w)\in\Si_0\times X'$ with $\norm{w}_{X'}\ll1$.  The Jacobian matrix of $G$ at the  trivial solution is given by
	$\di_{a^j}G_i\vert_{(0,0,0)}=\delta_{ij}$, since $\rho_{ij}(\si,w)\to 0$ for $w\to0$, as we have shown in the proof of \corref{cor3.2}. Combining these facts, the existence claim follows from Inverse Function Theorem.
	
	Estimate \eqref{a1} follows from \corref{cor3.2}. The Lipschitz estimate \eqref{a2} follows from the uniform estimate assumptions \eqref{f2}, \eqref{h2}, and the Lipschitz estimate assumption on the linearity, \eqref{2.17}.
\end{proof}

\section{Fixed point scheme}\label{sec:4}
In this section, we present a fixed point scheme to remove the destabilizing effect of the unstable modes \eqref{usm} of the linearized operator. This amounts to a second order correction in the sense of \eqref{6.2}-\eqref{6.3} below. As we show in \propref{prop4.1}, this fixed point scheme, together with the modulation from \secref{sec:3}, yields global existence as well as the dissipative estimate in the main result, \thmref{thm1}. 

Let 
\begin{equation}
	\label{4.1}
	\Xc:=C^1 (\Rb_{\ge0},\Si\times X)\cap Lip (\Rb_{\ge0},\Si\times X').
\end{equation}
Fix $0<\delta\ll1$. Let
\begin{equation}
	\label{A}
	\Ac_\delta=\Set{(\si,w)\in\Xc: \norm{w(t)}_{X'}\le \delta\br{t}^{-\al},\norm{\dot\si(t)}_Y\le c'(c_h+1) \delta\br{t}^{-\al}}. 
\end{equation}
Here $\al$ is the exponent in the propagator estimate (C2), $c'$ is the prefactor in \eqref{a2}, and $c_h$ is from \eqref{h2}.

The space $\Ac_\delta$ is non-empty, since the path $\si\equiv \si_0,\,w\equiv 0$ lies in it. Notice that this path corresponds to the trivial static solution to \eqref{1} given by $f(\si_0)$, see \eqref{2.6}. Now we show that $\Ac_\delta$ can be turned into a Banach space. 

\begin{lemma}\label{lem4.1}
	For every $\al\ge1$ and $c=c'(c_h+1)>0$, the function 
\begin{equation}\label{norm}
	\norm{(\si,w)}=\sup_{t\ge0}(\br{t}^\al\norm{w(t)}_{X'}+ c^{-1}t^{-1}\int_0^t\br{t'}^{\al} \norm{\dot\si(t')}_Y\,dt')
\end{equation}
	defines a norm on $\Xc$. Moreover, $(\Ac_\delta,\norm\cdot)$ is complete.

\end{lemma}
	\begin{proof}For $c>0$, \eqref{norm} defines a faithful sublinear form on $\Xc$. For $\al\ge1$, this \eqref{norm} is  stronger than the uniform norm on $\Xc$. Hence $(\Xc,\norm{\cdot})$ is complete. Direct calculation shows that $\Ac_\delta\subset \Xc$ is a closed subset of the unit ball of size $\delta$ w.r.t. the topology induced by \eqref{norm}, and it follows that $(\Ac_\delta,\norm{\cdot})$ is also complete.
	\end{proof}

Fix a path  $(\sn,\wn)\in\Xc$. Consider the following linear evolution equation for a path $\si(t)\in\Si,\,t\ge0$:
\begin{equation}
	\label{4.2}
	\dot \si=a^j(\sn,\wn)\xi_j(\sn),
\end{equation}
where $a^j=a^j(\sn,\wn)$ is the unique solution to the equation
$$
	G(a,\sn,\wn)=0,
$$
c.f. \eqref{3.12}. 
This is well defined for sufficiently small $\delta$ by \propref{prop3.3}. 

Couple \eqref{4.2}  to the following linear evolution equation for $w=w(t)\in X',\,t\ge0$:
\begin{equation}
	\label{4.3}
	\dot w = -L(\sn)w-N(\sn,\wn)- a^j(\sn,\wn) df(\sn)\xi_i(\sn).
\end{equation}
Since, by assumption, $L(\sn)$ is self-adjoint, \eqref{4.3} is globally well-posed for $X'$-small initial data by standard semiflow theory. 

We associate the following initial configurations to the system
\eqref{4.2}-\eqref{4.3}:
\begin{align}
	\label{4.4}
	\si(0)=&\si_0\in\Si_0,\\
	\label{4.5}
	w(0)=&\eta + \beta^{j_1}\phi_{j_1}(\si_0)+\ldots + \beta^{j_N}\phi_{j_N}(\si_0),
\end{align}
where  $\eta\in \Sc(\si_0)$,  $1\le j_p\le m_p$ with $m_p$ denoting the multiplicity of the eigenvalue $E_p<0$ from \eqref{C2}, various $\phi$'s are the unstable modes from \eqref{usm}, and various $\beta$'s are some real numbers. 

Fix $\si_0\in\Si_0$. For every $\eta\in\Sc(\si_0)$, and  $\beta^{j_p}=\beta^{j_p}(\eta)$ to be determined in \secref{sec:5} (see \eqref{5.18}), define a key solution map
\begin{equation}
	\label{Psi}
	\fullfunction{\Psi(\eta)}{\Ac_\delta\subset \Xc}{\Xc}{(\sn,\wn)}{\text{ the unique solution to }\eqref{4.2}-\eqref{4.5}.}
\end{equation}
This $\Psi$ is well-defined  by the preceding discussion regarding the global well-posedness of \eqref{4.2}-\eqref{4.3}.

\begin{proposition}
	\label{prop4.1}
	Fix $\si_0\in\Si_0$. Suppose $\Psi(\eta),\,\eta\in\Sc(\eta_0)$ has a fixed point in $\Ac_\delta$, denoted by $(\si,w)$. Then $u=f(\si)+w$ is a global dissipating solution to \eqref{1}, in the sense that $u(t)$ converges to the static solution $f(\si_0)$ as $t\to \infty$.
\end{proposition}
\begin{proof}
	If the expansion \eqref{3.6} holds, then $u=f(\si)+w$ solves \eqref{1} if and only if $w$ solves the under-determined equation
	\begin{equation}\label{4.6}
		\dot w = -E(f(\si)+w)-a^jdf(\si)\xi_j(\si).
	\end{equation} 
If $w$ comes from a path in $\Ac_\delta$, then $\norm{w}_{X'}\le\delta\ll1$ for all time, and therefore
\eqref{3.6} is valid.  If, moreover, $w$ comes from a fixed point of $\Psi(\eta)$, then \eqref{4.6} is satisfied, see the definition of \eqref{Psi} and \eqref{4.3}. This proves that $u=T(\si)+w$ is indeed a global solution to \eqref{1}. As for the dissipation, we use the fact that the definition of $\Ac_\delta$ implies $\norm{w}_{X'}\le\delta\br{t}^{-\al}\to0$ as $t\to 0$.  
\end{proof}

By \propref{prop4.1} above, Items 1-2 in \thmref{thm1} are proved once we establish the contraction property of $\Psi(\eta)$. This is the goal of the next section.

\section{Second order correction}\label{sec:5}
In this section, we define the coefficients $\beta^{j_p}(\eta)$ from \eqref{4.5}. If we view the removal of zero-unstable modes from $\eta\in\Sc(\si_0)$ (see \eqref{2.14}) as a first order correction in the initial perturbation \eqref{4.5}, then the result in this section amounts to a second order correction. Whereas the modulation method from \secref{sec:3} eliminates the zero modes from the evolution for all subsequent time $t>0$, the second order correction from this section effectively does the same with unstable modes \eqref{usm}.

\subsection{Initial choice of initial perturbations}\label{sec:5.1}
Throughout this subsection, fix a symmetry $\si_0\in\Si_0$, together with a vector $\eta\in\Sc(\si_0)$. From now on till the end of this section, dependence of various functions on $\si_0,\eta$ are not displayed but always understood, and all (explicit or implicit) constants are independent of $\si_0,\eta$ unless otherwise stated. 

\begin{theorem}
	\label{thm5.1}
	Let $\si_0\in\Si_0,\,0<\delta\ll1$. Suppose $\eta\in \Sc(\si_0)$ with $\norm{\eta}_X\le c_0\delta,\,c_0:=(4c_2)^{-1}$, where $c_0>0$ is as in \thmref{thm1},  $c_2>0$ as in \eqref{C2}. Let $(\sn,\wn)\in\Ac_\delta$ be a path satisfying initial condition \eqref{4.4}-\eqref{4.5}.
	
	Then there exist unique constants
	\begin{equation}
		\label{5.1}
		\beta^{j_p}(\sn,\wn)\quad (1\le j_p\le m_p,0\le p\le N),
	\end{equation}
in place of \eqref{4.5} s. th. $\Psi(\sn,\wn)\in \Ac_\delta$, where $\Psi=\Psi(\eta)$ is the map from \eqref{Psi}
\end{theorem}
\begin{proof}
	1. Write
		$$(\so,\wo):=\Psi(\sn,\wn).$$ 
	By \corref{cor3.2}, together with \eqref{4.2}, we have 
	$\norm{\di_t \so(t)}_Y=o\del{\norm{\wn(t)}_{X'}}$. For sufficiently small $\delta$ and $\norm{\wn(t)}_{X'}\le\delta\br{t}^{-\al}$, this gives $\norm{\so(t)}_Y\le \delta\br{t}^{-\al}$. Hence, it suffices to show that $\norm{\wo(t)}_{X'}\le\delta\br{t}^{-\al}$.
	
	2. Let $P_{j_p}(\sn(t)):X\to X,\,t\ge0$ be the orthogonal projection onto the span of unstable mode $\phi_{j_p}$. Let $P_S(\sn(t))=1-\sum _p\sum_{j_p} P_{j_p}(\sn(t))$ be the complement projection, c.f. \eqref{2.14}. 
	
	Denote by $w_\#=P_\#w,\#=j_p,S$ the projection of $w$ into various eigenspaces, and define the conjugated linearized operator $$L_S(\sn )=P_S(\sn)L(\sn )P_S(\sn).$$ 
	Since $L(\si)$ is self-adjoint, the various eigenvectors are mutually orthogonal. By this fact, we find that \eqref{4.3} is equivalent to the following system:
	\begin{align}
			\label{5.2}
			\dot w_S=&-L_S(\sn) w_S-P_SN(\sn,\wn),\\
			\label{5.3}
			\dot \beta^{j_p}=&-E_p\beta^{j_p} -\inn{N(\sn,\wn)}{\phi_{j_p}}.
				\end{align}
	Notice that zero modes are already eliminated by the modulation equation \eqref{4.2}, see \thmref{thm3.1}.
	
	We associate to \eqref{5.2}-\eqref{5.3} system the following initial configurations:
	\begin{align}
		\label{5.4}
		w_S(0)=&\eta\\
		\label{5.5}
		\beta^{j_p}(0)=&\beta^{j_p}(\sn,\wn),
	\end{align}
	where r.h.s. of \eqref{5.5} is to be determined.
	The system \eqref{5.2}-\eqref{5.3} is decoupled, so we solve the equations above separately. 
	
	3. First, for \eqref{5.2}, we use Duhamel's principle to rewrite
	\begin{equation}
		\label{5.6}
		w_S(t)=e^{-tL_S(\sn)}\eta+\int_0^t e^{-(t-t')}P_SN(\sn(t'),\wn(t'))\,dt'.
	\end{equation} 
	By the stability condition (C2), see \eqref{sec:2.2}, we have 
	\begin{equation}
		\label{5.7}
		\norm{w_S(t)}_{X'}\le c_2\del{\br{t}^{-\al}\norm{\eta}_X+\int_0^t\br{t'}^{-\al}\norm{N(\sn(t'),\wn(t'))}_X}\,dt'.
	\end{equation}
	By the assumption on the nonlinearity, and the choice of $\wn$ (see the definition of $\Ac_\delta$ in\eqref{A}), for every fixed $t\ge0$ there holds
	\begin{equation}
		\label{5.8}
		\norm{N(\sn(t),\wn(t))}_X=o(\norm{\wn(t)}_{X'})= o(\delta \br{t}^{-\al})\quad (\delta\to0).
	\end{equation}
	Since $\al\ge1$, \eqref{5.8} implies that for fixed $t$,
\begin{equation}\label{5.8'}
		\int_0^t\br{t'}^{-\al}\norm{N(\sn(t'),\wn(t'))}_X\,dt'= o(\br{t}^{-2\al+1}) =o(\br{t}^{-\al}).
\end{equation}
	Plugging this back to \eqref{5.7}, and using the assumption  $\norm{\eta}_X\le (4c_2)^{-1} \delta$, we find that for every $0<\delta\ll1$ there holds
	\begin{equation}
		\label{5.9}
		\norm{w_S(t)}_{X'}\le\frac{1}{2} \delta\br{t}^{-\al}.
	\end{equation} 

	4. Next, consider the equation \eqref{5.3}. The variation of parameter formula gives solution
		\begin{equation}
			\label{5.10}
			\beta^{j_p}(t)=e^{-\int_0^t E_p} \del{\beta(0)-\int_0^t \inn{N(\sn(t'),\wn(t'))}{\phi_{j_p}(t')} e^{\int_0^{t'} E_p}\,dt'},
		\end{equation}
	where $E_p=E_p(t),\,0\le p\le N$ denotes the $p$-th negative eigenvalue of $-L(\sn(t))$ with eigenvectors $\phi_{j_p}(t),1\le j_p\le m_p$. 
	
	Now, for the initial condition  \eqref{5.5}, we define
	\begin{equation}
		\label{5.11}
		\beta^{j_p}(\sn,\wn):=\int_0^\infty\inn{N(\sn(t'),\wn(t'))}{\phi_{j_p}(t')}e^{\int_0^{t'}E_p}\,dt'.
	\end{equation}
	 Since  $E_p(t)<0$ for all $t$, the integral on the r.h.s.  converges by the remainder estimate \eqref{5.8} for every $\al\ge1$. For later reference, here we record that \eqref{5.11} also depends on $\si_0,\,\eta$ through the initial condition on $(\sn,\wn)$, see \eqref{4.4}-\eqref{4.5}.
	
		We claim now that solution \eqref{5.10} satisfies \begin{equation}\label{5.11'}
			\beta^{j_p}(t)=o (\delta\br{t}^{-\al})\quad (\delta\to0)
		\end{equation} 
	for every fixed $t\ge0$, 
	if and only if the initial condition is given by \eqref{5.11}. 
	
		Suppose $\beta^{j_p}(t)\to0$. 
		Multiplying both sides of \eqref{5.10} by 
		$e^{\int_0^t E}$, and then taking $t\to \infty $, 
		we find that  $\beta^{j_p}(t)e^{\int_0^t E_p}\to0$ as $E_p(t)<0$,
		and therefore $\beta^{j_p}(0)$ is given by \eqref{5.11}.
		
		Conversely, if \eqref{5.11} holds, then the  formula \eqref{5.10} simplifies to 
		\begin{equation}
			\label{5.12}
			\beta^{j_p}(t)=\int_t^\infty \inn{N(\sn(t'),\wn(t'))}{\phi_{j_p}(t')}e^{\int_t^{t'}E_p}\,dt'.
		\end{equation}
	Let $$E_p:=\sup_{t\ge0}E_p(t).$$
	Then  $E_p<0$ by \eqref{C2}. 
	Hence, by \eqref{5.8},  we conclude from formula \eqref{5.12} that for each fixed $t\ge0$,
\begin{equation}
	\label{5.12'}
	\abs{\beta^{j_p}(t)}= o(\delta) \int_t^\infty {t'}^{-\al}e^{Et'}\,dt'\quad (\delta\to0). 
\end{equation}
Notice that in the integrand we have $t$ rather than $\br{t}$, due to the smoothing property of the exponential at $t=0$. 

	For fixed $\al\ge1$ and all $t\gg0$, r.h.s. above is bounded by 
	$$k(t):=\int_t^\infty {t'}^{-\al}e^{Et}.$$
	At $t=0$, this reduces to the well known exponential integral, see \cite{MR0167642}*{Chapt. 5}. In particular, for each fixed $\delta\ll1$ we have 
	\begin{equation}
		\label{5.12''}
		\abs{\beta^{j_p}(t)}\le k(t)=o(t^{-\al})\quad (t\to\infty). 
	\end{equation}
	This, together with \eqref{5.12'}, implies \eqref{5.11'}. 
		
	5. By \eqref{5.9}, \eqref{5.12'}-\eqref{5.12''},  we conclude that for every fixed $0<\delta\ll1$ and all $t\ge0$,
	\begin{equation}
		\norm{\wo(t)}_{X'}\le \norm{w_S(t)}_{X'}+\sum_p\sum_{j_p}\beta^{j_p(t)}\le \delta\br{t}^{-\al}.
	\end{equation}
	This, together with Step 1, proves the theorem.
\end{proof}

For simplicity, below we write elements in $\Xc$ as 
$$U=(w,\si),\quad U^{(n)}=(w^{(n)},\si^{(n)})\quad (n=0,1,\ldots).$$
\begin{proposition}
	\label{prop5.2}
	Fix $1\le j_p\le m_p,\,0\le p\le N$. Let $\beta\equiv \beta^{j_p}:\Xc\to\Rb$ be the functional defined in \eqref{5.10}. Then for every $\Un,\Uo\in\Ac_\delta$, there hold
	\begin{align}
		\label{5.13}
		\beta(\Un)=&o(\norm{\Un}),\\
		\label{5.14}
		\abs{\beta(\Un)-\beta(\Uo)}=&O(\delta\norm{\Un-\Un}).	
	\end{align}
\end{proposition}
\begin{proof}
	1. Since the index $j_p$ is fixed, in this proof the dependence on $j_p$ is not displayed. Various implicit constants below do not depend on $j_p$.
	
	2. By \eqref{5.11}, together with the spectral condition (C2),
	we have 
\begin{equation}\label{5.13'}
		\beta(\Un)=\int_0^\infty e^{E^{(0)}t'}o(\sup_{0\le t''\le t'}\norm{\wn(t'')}_{X'})\,dt',
\end{equation}
	where $E^{(0)}:=\sup_{t\ge0}E^{(0)}_p$ is strictly negative. The integrand in \eqref{5.13'} can be bounded as 
	$$e^{E^{(0)}t'}o(\sup_{0\le t''\le t'}\norm{\wn(t'')}_{X'})=  \norm{U^{(0)}}o(\br{t'}^{-\al}e^{Et'}).$$
	The integral of the r.h.s. over $t\ge0$, which is convergent since $E<0$ and $\al\ge1$, is of the order $o(\norm{\Un})$. This implies \eqref{5.13}.
	
	3. For $n=0,1$, define the linear functional
	 \begin{equation}\label{5.15}
		\fullfunction{l^{(n)}(t)}{X}{\Rb}{u}{\inn{u}{\phi^{j_p}(\si^{(n)}(t))}}.
	\end{equation}
	Then we have by \eqref{5.11} that
	\begin{equation}
		\label{5.16}
\begin{aligned}
	\abs{\beta(\Un)-\beta(\Uo)}\ls& \int_0^\infty\abs{(l^{(1)}(t')-l^{(0)}(t'))N(\Un(t'))}\,dt' \\&+ \int_0^\infty\abs{l^{(0)}(t')(N(\Uo(t'))-N(\Un(t')))}.
\end{aligned}
	\end{equation}
	The integrand in the first term on the r.h.s. can be bounded as \begin{equation}\label{5.16'}
		\abs{(l^{(1)}(t')-l^{(0)}(t'))N(\Un(t'))}=o(\delta\br{t'}^{-\al})\sup_{0\le t''\le t'}d(\so(t''),\sn(t''),
	\end{equation}
	by  estimate \eqref{5.8}. The second term in the r.h.s. of \eqref{5.16} is bounded as 
\begin{equation}\label{5.16''}
		\abs{l^{(0)}(t')(N(\Uo(t'))-N(\Un(t')))} =O(\delta\sup_{0\le t''\le t'}\br{t''}^{-\al}\norm{\wo(t'')-\wn(t'')}),
\end{equation}
	where we use estimate \eqref{2.17} on the nonlinearity, and the uniform bound 
	\begin{equation}\label{lu}
		\sup_{t\ge0}\norm{l^{(0)}(t)}_{X\to\Rb}\ls 1.
	\end{equation}
	To see \eqref{lu}, we first use  the definition of $\Ac_\delta$ and $\al\ge1$ to get $d(\sn(t),\sn(0))\ls1$. Then \eqref{lu} follows from this, together with condition  \eqref{C4}.
	
	Plugging \eqref{5.16'}-\eqref{5.16''} back to \eqref{5.16}, and using the definition \eqref{norm} as in Step 1,  we conclude \eqref{5.14} after integration.
\end{proof}

\subsection{Adaptive choice of initial perturbations}\label{sec:5.2}
			Fix a symmetry $\si_0\in\Si_0$ and  $\eta\in\Sc(\si_0)$. For $n=0$, we solve the Cauchy problem \eqref{4.2}-\eqref{4.5} with initial condition \eqref{4.5} given by \thmref{thm5.1}. Denote the solution by $\Uo(\eta)\equiv (\so(\eta),\wo(\eta)):=\Psi(\eta)(\sn,\wn).$ By \thmref{thm5.1}, we have $\Uo(\eta)\in\Ac_\delta$. Hence, we can repeat this process for $n=1,2,\ldots$ and a sequence $U^{(n)}(\eta)\in \Ac_\delta$. 
	\begin{lemma}\label{lem5.1}
		If $\Psi(\eta)$ is a contraction (see \eqref{Psi}), then
		the following limits exist:
		\begin{align}
			\label{5.17}
				U_\infty(\eta):=&\lim_{n\to\infty} U^{(n)}(\eta)\in\Ac_\delta,\\
				\label{5.18}
				\beta^{j_p}(\eta):=&\lim_{n\to\infty}\beta^{j_p}(\eta,U^{(n)})\quad (1\le j_p\le m_p,0\le p\le N).
		\end{align}
	\end{lemma}
The path $U_\infty$ is the unique fixed point of $\Psi(\eta)$ in $\Ac_\delta$.
\begin{remark}
	Notice that \eqref{5.18} depend on $\eta$ only, in contrast to the initial choice \eqref{5.11}. 
\end{remark}
	\begin{proof}
		If the limit \eqref{5.17} exists, then it is the unique fixed point of $\Psi(\eta)\in\Ac_\delta$ by construction.  
	If $\Psi(\eta)$ is a contraction, then the sequence $U^{(n)}(\eta)$ is Cauchy in $\Ac_\delta$ by \thmref{thm5.1}. By \lemref{lem4.1}, $(\Ac_\delta,\norm{\cdot})$ forms a Banach space. By the Lipschitz estimate \eqref{5.14}, for each $j_p$ and $\beta^{j_p}(\eta,U)$ as in \propref{prop5.2},  the number sequence $\beta^{j_p}(\eta, \Un),\beta^{j_p}(\eta,\Uo),\ldots$ is also Cauchy, and therefore converges in $\Rb$. 
\end{proof}

Fix $\si_0\in\Si_0$ and a number $0<\delta\ll1$.  Define the key map 
\begin{equation}
	\label{Phi}
	\fullfunction{\Phi}{\Sc(\si_0)\cap B_\delta}{X'}{\eta}{\beta^{j_1}(\eta)\phi_{j_1}(\si_0)+\ldots+\beta^{j_N}(\eta)\phi_{j_N}(\si_0)}.
\end{equation}
This map $\Phi$ depends on the fixed symmetry frame $\si_0$, as well as the number $\delta$, which goes into the definition of $\Psi$ through the space $\Ac_\delta$. To be consistent in notation, we do not display the dependence on the fixed small number $\delta$. 

If $\eta=0$, then $\Phi(\eta)=0$, since in this case we have the trivial solution $u\equiv f(\si_0)$ is a global solution to \eqref{1} satisfying \eqref{4.4}-\eqref{4.5}. This, together with \defnref{sm}, implies $\Mc\supset \Mc_\delta=f(\Si_0).$ 

Suppose  the initial condition \eqref{4.5} is given by $\eta+\Phi(\eta)$. Then by construction, the limit $U_\infty(\eta)$ is the unique fixed point of $\Psi(\eta)$. By \eqref{prop4.1}, this implies Items 1-2 in the main result, \thmref{thm1}.

It remaining of this subsection, we prove $\Psi(\eta)$ is a contraction. As in \secref{sec:5.1}, since $\si_0,\eta$ is fixed, dependence of various functions on $\si_0,\eta$ are not displayed but always understood, and all (explicit or implicit) constants are independent of $\si_0,\eta$. 

\begin{theorem}
	\label{thm5.2}
	Let $\si_0,\,\eta$ be as in \thmref{thm5.1}. Then $\Psi=\Psi(\eta)$ is a contraction on $\Ac_\delta$. 
\end{theorem}
\begin{proof}
	1. Take two paths $U^n\in\Ac_\delta,\,n=0,1$. Let $U^{n+2}:=\Psi(U^n)$. Then $U^{n+2}\in\Ac_\delta$ by \thmname{thm5.1}. By the definition of $\Psi$ in \eqref{Psi}, the difference $U^3-U^2$ solves the linear system
	\begin{align}
		\label{5.19}
		\begin{split}
			\di_t(w^3-w^2)=&-L(\si^0)(w^3-w^2)+V(\si^0,\si^1)w^3\\&-df(\si^0)(\dot \si^3-\dot \si^2)-(df(\si^1)-df(\si^0))\dot\si^3\\&-N(U^1)+N(U^0),
		\end{split}\\
			\label{5.23'}
	V(\si^0,\si^1):=&L(\si^1)-L(\si^0):X'\to X,\\
		\label{5.20}
		(w^3-w^2)\vert_{t=0}=&(\beta^{j_p}(U^1)-\beta^{j_p}(U^0))\phi_{j_p}(\si_0).
	\end{align}

	2. We study \eqref{5.19}-\eqref{5.20} as in the proof of \thmref{thm5.1}. Let 
\begin{align}
	\label{5.21}
		&P_S(t):=\text{stable projection onto }\Sc(\si^0(t))\subset X,\\
		&L_S=L_S(t):=P_S(t)L(\si^0(t))P_S(t),\\
		\label{5.22}
		&l_{j_p}(t):u\mapsto \inn{u}{\phi^{j_p}(\si^0(t))}\text{as in \eqref{5.15}}.
\end{align}
	Projecting \eqref{5.19}-\eqref{5.20} onto the eigenspaces of $L(\si^0)$, we find that this system is equivalent to
	\begin{align}
		\label{5.23}
		\begin{split}\di_tP_S(w^3-w^2)=&-L_SP_S(w^3-w^2)\\&+P_S(V(\si^0,\si^1)w^3 -(df(\si^1)-df(\si^0))\dot\si^3-N(U^1)+N(U^0)),
		\end{split}\\
		\label{5.24}
		\dot \g^{j_p}=&-E_p\g^{j_p}-f^{j_p}(U^0,U^1,U^3),\\
		\label{5.25}
		f^{j_p}(U^0,U^1,U^3):=&l^{j_p}(V(\si^0,\si^1)w^3-(df(\si^1)-df(\si^0))\dot\si^3-N(U^1)+N(U^0)),\\
		\label{5.26}
		P_S(w^3-w^2)\vert_{t=0}=&0,\\
		\label{5.27}
		\g^{j_p}(0)=&\beta^{j_p}(U^1)-\beta^{j_p}(U^0).
	\end{align}
	Here $\g^{j_p}:=l^{j_p}(w^3-w^2)$. To get \eqref{5.23}-\eqref{5.27}, we use the orthogonality among various eigenvectors of the self-adjoint operator $L(\si^0)$.
	
	3. First, consider evolution of the stable mode, \eqref{5.23}. By Duhamel's principle and initial condition \eqref{5.27}, we have
	\begin{equation}
		\label{5.28}
		\begin{split}
			&P_S(w^3-w^2)\\=&\int_0^t e^{-(t-t')L_S(t')}\underbrace{P_S(V(\si^0,\si^1)w^3 -(df(\si^1)-df(\si^0))\dot\si^3-N(U^1)+N(U^0))}_{\text{all depend on }t'}\,dt'.
		\end{split}
	\end{equation}
	
	For any element $w\in X'$ with $\norm{w}_{X'}\ll1$, we write
\begin{equation}\label{5.29}
		V(\si^0,\si^1)w=(E'(f(\si^1)+w)-E'(f(\si^0)+w))-(N(\si^1,w)-N(\si^0,w)).
\end{equation}
	Since $E$ is $C^2$ on $U$, $E'$ is $C^1$ from $X'\to X$. Hence, the first term on the r.h.s. can be bounded as 
	$$\norm{E'(f(\si^1)+w)-E'(f(\si^0)+w)}_X\le c_1 \norm{f(\si^1)-f(\si^1)}_{X'}\le c_1d_f d(\si^0,\si^1). $$
	For the first inequality, we use the uniform bound assumption (C1) and the mean value theorem for Fr\'echet differentiable maps. For the second  inequality, we use the assumption \eqref{f2} on the parametrization $f:\Si\to X'$. Similarly, by assumption \eqref{2.17}, the second term in the r.h.s. of \eqref{5.29} can be bounded by $d(\si^0,\si^1)$ as well. 
	Hence, since $U^3\in\Ac_\delta$, we find 
	\begin{equation}
		\label{5.30}
		\norm{V(\si^0(t),\si^1(t))w^3(t)}_X\le c_1c_f \delta\br{t}^{-\al}d(\si^0,\si^1).
	\end{equation}
	
	Next, since $f$ is $C^2$ on $\Si_0$, we have $\norm{df(\si^1)-df(\si^0)}_{Y'\to X'}\ls d(\si^0,\si^1)$. This, together with the fact that $U^3\in\Ac_\delta$, implies
	\begin{equation}
		\label{5.30'}
		\norm{(df(\si^1)-df(\si^0))\dot\si^3}_{X}\le c_f\delta\br{t}^{-\al}d(\si^0(t),\si^1(t)).
	\end{equation}
	
	Using \eqref{5.30}-\eqref{5.30'} and another application of \eqref{5.29}, we conclude the Lipschitz estimate
	\begin{equation}
	\label{5.31}
	\begin{aligned}
		&\norm{V(\si^0(t),\si^1(t))w^3(t)-(df(\si^1(t))-df(\si^0(t)))\dot\si^3(t))- N(U^1(t))+N(U^0(t))}_X\\
		\le&\delta(c_f(c_1+1)\br{t}^{-\al}+1)(d(\si^0(t),\si^1(t))+\norm{w^1(t)-w^0(t)}_{X'})\\
		\le& \delta \norm{U^1-U^0}\br{t}^{-\al}(c_f(c_1+1)\br{t}^{-\al}+1)
	\end{aligned}
	\end{equation}
	where in the last step we use the definition \eqref{norm}.
	Plugging this back to \eqref{5.28}, and then using the propagator estimate \eqref{C3}, we conclude 
	\begin{equation}
		\label{5.32}
		\norm{P_S(w^3-w^2)}_X\le \delta  \norm{U^1-U^2}c_\al(t),
	\end{equation}
	where, with $c_2$ denoting the constant from \eqref{C3}, 
	$$c_\al(t):=c_2\int_0^t\br{t}^{-2\al}(c_f(c_1+1)\br{t}^{-\al}+1)\,dt.$$
	This integral is positive and of the order $O(\br{t}^{-\al})$ for $\al\ge1$.
	
	4. Next, fix $j_p$ and consider the ODE \eqref{5.24}. For simplicity of notation, in this step we drop the dependence on $p$ and $j_p$. All implicit constants are independent of $p,\,j_p$.
	
	Since $w^3(t),w^2(t)$ both tends to $0$ in $X'$ by the initial choice of $\beta$ in \thmref{thm5.1}, the function $\g(t)=l(\si^0(t))\to0$ as $t\to\infty$, Hence, as in Step 4 of \thmref{thm5.1},	we find that
	\begin{equation}
		\label{5.33}
		\abs {\g(t)}\le \int_t^\infty f(U^0(t'),U^1 (t'),U^3 (t'))e^{Et'}\,dt',
	\end{equation}
where $ E:=\sup_{t\ge0}E(t)<0$ by \eqref{C2}.

	Using estimate \eqref{5.31}, together with the definition \eqref{5.27}, we find that 
	\begin{equation}
		\label{5.34}
		\abs {\g(t)}\le\delta\norm{U^1-U^0} c_\al'(t),
	\end{equation}
	where
	$$c_\al'(t):=\int_t^\infty \br{t}^{-\al}(c_f(c_1+1)\br{t}^{-\al}+1) e^{Et}\,dt.$$
	This integral is finite, and $O(\br{t}^{-\al})$ for $\al\ge1$ and $E<0$.
	
	 5. Combining \eqref{5.34} and \eqref{5.32} gives the estimate \begin{equation}
	 \label{5.35}
	 	\norm{w^3-w^2}_{X'}\le C_\al\delta \norm{U^1-U^0}\del{(\sum_{p=1}^N m_p)+1}\br{t}^{-\al},
	 \end{equation}
	 where $C_\al>0$ depends on the decay property of $c_\al(t)+c_\al'(t)$ as $t\to\infty$ (this limit exists and is finite by the large time asymptotics on $c_\al(t),c_\al'(t)$).  
	 
	6. Next, consider the equation for $\dot\si ^3-\dot\si^2 $, which follow from the definition of $\Psi$ in \eqref{Psi} and \eqref{4.2},\eqref{4.6}:
	\begin{align}
		\label{5.36}
		\di_t(\si^3-\si^2)=&a^j(U^1)\xi_j(\si^1)-a^j(U^0)\xi_j(\si^0),\\
		\label{5.37}
		 (\si^3-\si^2)\vert_{t=0}=&0.
	\end{align}
	Write
	\begin{equation}
		\label{5.38}
		a^j(U^1)\xi_j(\si^1)-a^j(U^0)\xi_j(\si^0)=a^j(U^1)(\xi_j(U^1)-\xi_j(U^0))+(a^j(U^1)-a^j(U^0))\xi_j(U^0).
	\end{equation}
	By \corref{cor3.2}, the choice $U^1\in\Ac_\delta$, and the assumption \eqref{h2}, for  fixed $\delta\ll1$, the first term in the r.h.s. of \eqref{5.38}  is bounded as $$\norm{a^j(U^1(t))(\xi_j(U^1(t))-\xi_j(U^0(t)))}_Y\le c_h\delta\br{t}^{-\al}\norm{U^1-U^0}\quad(t\to\infty).$$
	By the Lipschitz estimate \eqref{a2}, similarly, the second term in the r.h.s. of of \eqref{5.38} is bounded as
	$$\norm{(a^j(U^1(t))-a^j(U^0(t)))\xi_j(U^0(t))}_Y\le c'c_h\delta\br{t}^{-\al}\norm{U^1-U^2}.$$
	Plugging the two preceding estimates back to \eqref{5.36}, we find 
	\begin{equation}\label{5.39}
		\norm{\di_t(\si^3-\si^2)}_Y\le (c'(c_h+1))\delta \br{t}^{-\al}. 
	\end{equation}

	7. Combining \eqref{5.35} and \eqref{5.39}, we conclude that if we choose
	$$\delta_0:=\frac{1}{2}\min\del{ \frac{1}{C_\al \del{(\sum_{p=1}^N m_p)+1}},\frac{1}{c'(c_h+1)}}$$
	 then for every $0<\delta\le\delta_0$, there holds
	$$\norm{U^3-U^2}\le \frac{1}{2}\norm{U^1-U^0}.$$
	This shows that $\Psi$ is a contraction so long as $\delta\ll1$.
\end{proof}

\section{Stable manifold}\label{sec:6}
Recall that in \defnref{sm}, we have defined the set
\begin{equation}\label{6.1}
	\Mc=\Mc(\si_0,\delta):=\Set{f(\si_0)+\eta+\Phi(\eta):\eta\in\Sc(\si_0),\,\norm{\eta}_X'<\delta,\,\Phi\text{ as in \eqref{Phi}}}.
\end{equation}
By \propref{prop4.1} and \thmref{thm5.2}, each vector $u_0\in\Mc$ generates a global dissipating solution $u\in\Xc$ (see \eqref{4.1}) to \eqref{1} with $u(0)=u_0$. Moreover, the evolution of this $u$ is essentially governed by an equation in the manifold of static solutions $f(\Si)$, namely \eqref{2.23}.

In Sections \ref{sec:4}-\ref{sec:5}, for the most parts we have been considering the property of the $\Psi(\eta)$ map, given in \eqref{Psi}, with fixed $\eta$. In order to justify the manifold structure on $\Mc$, we need to establish suitable smallness estimates for the map $\Phi$ from \eqref{Phi} in terms of $\eta$. This is the main goal of this section, corresponding to the claimed inequalities \eqref{2.19}-\eqref{2.20} in the main theorem. 
\begin{theorem}
	\label{thm6.1}
For every $\eta_0,\,\eta_1\in\Sc(\si_0)$ with $\norm{\eta_0}+\norm{\eta_1}\le\delta$, there hold
	\begin{align}
		\label{6.2}
		\norm{\Phi(\eta_0)}_{X'}=&o(\norm{\eta_0}_{X'})\quad (\delta\to0),\\
		\label{6.3}
		\norm{\Phi(\eta_0)-\Phi(\eta_1)}_{X'}=&O(\delta\norm{\eta_0-\eta_1}_{X'})\quad(\delta\to0),
	\end{align}
c.f. \eqref{2.19}-\eqref{2.20}. 
\end{theorem}
\begin{remark}
	Notice the difference between these and \propref{prop5.2}. Here $\eta$ is allowed to vary in a small ball in $\Si(\si_0)$.
\end{remark}
\begin{proof}[Proof of \thmref{thm6.1}]
		1. We first prove \eqref{6.2}. Take $\eta_0\in \Sc(\si_0)$ with $\norm{\eta_0}_{X'}\le\delta\ll1$.Fix some $0<\eps\ll1$ to be determined. As in the beginning of \secref{sec:5.2}, we iterate $\Psi=\Psi(\eta_0)$ for 
		$m$ times on a fixed $U^{(0)}\in \Ac_\delta$ to get a sequence of paths  $U^{(m)}\in \Ac_\delta$. For each fixed $j_p$, we also geta sequence of numbers  $\beta^{(m)}:=\beta^{j_p}(\eta_0, U^{(m-1)})\to \beta:=\beta^{j_p}(\eta_0)$, where for the transparency of notation the superindex on $j_p$ is dropped. 
		
	By \lemref{lem5.1} and \thmref{thm5.2}, for every $\eps>0$ there exists $m=m(\eps)\gg1$ s.th.
	s.th. 
	\begin{equation}
		\label{6.2'}
		\abs{\beta^{(m)}-\beta}<\eps.
	\end{equation}

	Consider the path $w^{(m+1)}$ with $m=m(\eps)\gg1$. This solves the Cauchy problem
	\begin{align}
		\label{6.3'}
		\dot w^{(m+1)}=&-L(\si^{(m)})w^{(m+1)}- N(\si^{(m)},\xi^{(m)})-a^j(U^{(m)})df(\si^{(m)})\xi_j(\si^{(m)}),\\
		\label{6.4}
		w^{(m+1)}(0)=& \eta_0+\beta_{j_1}^{(m)}\phi_{j_1}(\si_0)+\ldots+\beta_{j_N}^{(m)}\phi_{j_N}(\si_0)
	\end{align}
	Using \eqref{5.7}-\eqref{5.8'},  together \eqref{5.11'}, we find that for every fixed $t$, there holds 
	\begin{equation}\label{6.5}
		\norm{w^{(m+1)}(t)}_{X'}\le 2c_2\norm{\eta_0}_{X'}+o(\delta\br{t}^{-\al})\quad(\delta\to0).
	\end{equation}	Here $c_2$ is the constant in the propagator estimate \eqref{C3}.

We claim now
\begin{equation}\label{6.6}
		\beta^{(m+1)}= o(\norm{\eta_0}_{X'})\quad(\delta\to0).
\end{equation}
	Recall that by assumption, $\norm{\eta_0}\ls\delta$. If $\norm{\eta_0}\sim\delta$, then \eqref{6.6} follows from \eqref{5.13} and  estimate \eqref{6.5}. If $\norm{\eta_0}\sim\delta_0\ll1$, then we can repeat the argument in \secref{sec:5} with $\delta_0$ in place of $\delta$ everywhere, to obtain \eqref{6.8} with $\delta_0$ in place of $\delta$. This proves the claim.
	
Estimate \eqref{6.6}, together with \eqref{6.2'}, gives 
 $$\beta\le\eps+ 3c_2\norm{\eta_0}_{X'}+\eps.$$
Since $\eps$ is arbitrary, this implies \eqref{6.2}.

	2. Next, we prove \eqref{6.3}. Recall the map $\Psi(\eta),\,\eta\in\Sc(\si_0)$ is defined in \eqref{Psi}. For every fixed $U\in\Ac_\delta$, we claim the following  Lipschitz estimates where the implicit constants are independent of $U$:
	\begin{align}
		\label{6.7}
		\norm{\Phi(\eta_0)-\Phi(\eta_1)}_{X'}&\ls \delta \lim_{m\to\infty}\norm{\Psi^{(m)}(U,\eta_0)-\Psi^{(m)}(U,\eta_1)}\\
		\label{6.8}\norm{\Psi(U,\eta_0)-\Psi(U,\eta_1)}&\ls \norm{\eta_0-\eta_1}_{X'}.
	\end{align}
	Here the notation $\Psi^{(m)}$ means iterating the map $m$ times.Notice that the limit in \eqref{6.7} exists, since by 	\thmref{thm5.2}, the sequence  $\Psi^{(m)}(U,\eta_n),\,n=0,1$ converges as $M\to\infty$ to the fixed point of $\Psi(\eta_n)$ in the space $\Ac_\delta$.
	
		3. We first prove \eqref{6.7}, \textit{assuming \eqref{6.8} holds}. We do this by adapting the construction from Step 1 for $n=0,1$. 
		
	As in Step 1, for $n=0,1$ and every fixed $j_p$,  put \begin{align}
		&\label{Um}U^{(0)}\in\Ac_\delta,\quad U^{(m)}_n=\Psi(\eta_n)(U^{(m-1)}),\quad m=1,2,\ldots,\\
		&\label{betam}\beta^{(m)}_n:=\beta^{j_p}(\eta_n, U^{(m-1)}),\quad  \beta_n:=\lim_{m\to\infty}\beta^{(m)}_n.
	\end{align}
The function $\beta^{j_p}$ in \eqref{betam} is defined in \eqref{5.11}.

	For every $0<\eps\ll1$, we choose some $m=m(\eps)\gg1$ s.th.
	\begin{equation}\label{6.9}
			\abs{\beta_i(\eta_1)-\beta_i(\eta_0)}\le  \eps+\int_0^\infty \abs{f_1(U^{(M)}(t)-f_0(t)},
	\end{equation}
	where, for $n=0,1,$, $m=1,2,\ldots$, and $t\ge0$, we put
	\begin{equation}
		\label{6.10}
		\fullfunction{f_n^{(m)}}{ \Rb_{\ge0}}{\Rb}{t}{\inn{N(U_n^{(m)}(t))}{\phi^{j_p}(\si^{(m)}(t))}}.
	\end{equation}
	Here $f_n^{(m)}$ depends on $\eta_n$ through definition of $U^{(m)}_n$, see\eqref{Um}. 
	
	The claim now is that there exists $C>0$ independent of $m,t$ s.th. 
	\begin{equation}\label{6.11}
		\abs{f_1^{(m)}(t)-f_0^{(m)}(t)}\le C \delta\br{t}^{-\al}\norm{U^{(M)}_n-U^{(M)}_n}.
	\end{equation}
	If \eqref{6.11} holds, then plugging it into \eqref{6.9}
	and taking $\eps\to 0,\,m(\eps)\to\infty$ gives \eqref{6.7}.

	In view of the definition of $f_n^{(m)}$ from \eqref{6.10} and the Lipschitz
	estimate \eqref{2.17}, to get \eqref{6.11},
	it suffices to show that the paths 
	$U^{(m)}_0,\,U^{(m)}_1$ remains uniformly 
	close in $\Ac_\delta$ for all large $m$.
	The latter follows from the definition \eqref{Um}, and 
	the uniform bound \eqref{6.8}, provided $\eta_0$ 
	and $\eta_1$ are close (which is indeed the case, since $\norm{\eta_0}+\norm{\eta_1}\ll\delta$).
	This proves \eqref{6.7}.
	
	4. It remains to prove \eqref{6.8}. 
	Fix a path $U=(\si,w)\in\Ac_\delta$, and let
	$$(\si^n,w^n):=\Psi(U,\eta_n),\quad n=0,1.$$
	By the definition of $\Psi$ from \eqref{Psi}, we find that the difference $w^1-w^0$ satisfies
	the Cauchy problem
	\begin{align}
		\label{6.12}
		\di_t(w^1-w^0)-L(a)(w^1-w^0)&=0,\\
		\label{6.13}
		\begin{split}
			(w^1-w^0)\vert_{\tau=0}&=\eta_1-\eta_0\\&+\sum_p(\beta^{j_p}(\eta_1,U)\phi_{j_p}(\si)-\beta^{j_p}(\eta_0,U)\phi_{0_p}(\si)
		\end{split}\\
		\label{6.14} 
		\di_t(\si^1-\si^0)&=0,\\
		\label{6.15}
		(\si^1-\si^0)\vert_{\tau=0}&=0.
	\end{align}

	\eqref{6.14}-\eqref{6.15} implies $\si^1(t)-\si^0(t)\equiv0$.
	Hence it suffices to study \eqref{6.12}-\eqref{6.13} only. As in the proof of \thmref{thm5.1}. Indeed, if we decompose $w^1(t)-w^0(t)$ in terms of the eigenfunctions of $L(\si(t))$, then the projection of this difference along the stable modes of $L(\si(t))$ satisfies an decay estimate of the form \eqref{5.7}, and the projections along the unstable	modes satisfy an remainder estimate of the form
	\eqref{5.11'}. The latter is due to the choice of initial condition \eqref{6.13}, as we have shown in the proof of \thmref{thm5.2}. We conclude estimate 
	$$\norm{w^1(t)-w^0(t)}_{X'}\ls \br{t}^{-\al}\norm{\eta_1-\eta_0}_{X'}.$$
	This proves \eqref{6.8}.
\end{proof}

		\section*{Acknowledgment}
		The Author is supported by Danish National Research Foundation grant CPH-GEOTOP-DNRF151 and The Niels Bohr Grant from the Royal Danish Academy of Sciences and Letters. The Author thanks IM Sigal for support and hospitality in the completion of this work during a visit to the University of Toronto.   
		
		\section*{Declarations}
		\begin{itemize}
			\item Conflict of interest: The Author has no conflicts of interest to declare that are relevant to the content of this article.
			\item Data availability: Data sharing is not applicable to this article as no datasets were generated or analysed during the current study.
		\end{itemize}

		\appendix
		\section{Basic variational calculus}
		Here we recall some basic elements of variational calculus that have been used 
		repeatedly. For details, see for instance \cite{MR2431434}*{Appendix C}, \cite{MR1336591}*{Chapt. 1}.
		
		\subsection{Fr\'echet Derivative} Let $X,\,Y$ be two Banach spaces. Let $U$ be an open set in $X$.
		For a map $g:U\subset X\to Y$ and a vector $u\in U$,
		the Fr\'echet derivative $dg(u)$ is a linear map from $X\to Y$ s.th.
		$g(u+v)-g(u)-dg(u)v=o(\norm{v}_X)$ for every $v\in X$ with $\norm{v}_X\ll1$ . 
		If $dg(u)$ exists at $u$, then it is unique. If $dg(u)$ exists for every $u\in U$,
		and the map $u\mapsto dg(u)$ is continuous from $U$ to the space of linear operators $L(X,Y)$, then we
		we say $g$ is $C^1$ on $U$. 
		In this case, $dg(u)$ is uniquely given by 
		$$v\mapsto {\od{g(u+tv)}{t}}\vert_{t=0}\quad (v\in X).$$
		Iteratively, we can define higher order derivatives this way.

		\subsection{Gradient and Hessian} If $X$ is a Hilbert space over a scalar field $Y$, then by Riesz representation,
		we can identify $dg(u)$ as an element in $X$, denoted by $\Si(u)$. The vector
		$\Si(u)$ is called the $X$-gradient of $g$. 
		Similarly, we denote $\Si'(u)$ the second-order Fr\'echet derivative
		$d^2g(u)$. If $g$ is $C^2$, then $\Si'$ can be identified as a symmetric 
		linear operator
		uniquely determined  by the relation  
		$$\inn{\Si'(u)v}{w}={\md{g(u+tv+sw)}{2}{t}{}{s}{}}\vert_{s=t=0}\quad (v,w\in X).$$
		
		\subsection{Remainder and Composition} Let $X$ be a Hilbert space over a scalar field $Y$. Suppose $g$ is $C^2$ on $U\subset X$. Define a scalar function $\phi(t):=g(v+tw)$ for vectors
		$v,w$ s.th. $v+tw\in U$ for every $0\le t\le 1$.
		Then the elementary Taylor expansion at $\phi(1)$ gives 
		$$g(v+w)=g(v)+\inn{\Si(v)}{w}+\frac{1}{2}\inn{\Si'(v)w}{w}+o(\norm{w}_X^2).$$
		Here we have used the definition of $\Si$ and $\Si'$ from the last subsection.
		
		Let $\Omega\subset \Rb^d$ be a bounded domain with smooth boundary.
		Fix $r>d/2,\,f\in C^{r+1}(\Rb^n)$. For $u:\Omega\to \Rb^n$,
		define a map $g:u\mapsto f\circ u$. Then $g:H^r(\Omega)\to H^r(\Omega)$
		is $C^1$, 
		and the Fr\'echet derivative is given by $v\mapsto \grad f\cdot v$.

		\bibliography{bibfile}
	\end{document}